\documentclass[leqno,11pt]{amsart}
\pdfoutput=1		

\usepackage[bookmarksnumbered]{hyperref}	
\usepackage[ngerman,USenglish]{babel}			
\usepackage{color}											
\usepackage[ascii]{inputenc}								
\usepackage{aliascnt}										
\usepackage{microtype}									
\usepackage[all]{hypcap}
\usepackage{tikz}												
\usepackage{verbatim}

\newcommand{\newjointcountertheorem}[3]{
	\newaliascnt{#1}{#2}
	\newtheorem{#1}[#1]{#3}
	\aliascntresetthe{#1}	
}

\newtheorem*{thm-fusion}{Theorem \ref{thm:central-limit-basic-spec-fusion}}
\newtheorem*{thm-fusion-strings}{Theorem \ref{thm:central-limit-strings-fusion}}
\newtheorem*{cnj-fusion}{Conjecture \ref{cnj:fusion}}

\newtheorem{thm}{Theorem}[section]
\newjointcountertheorem{satz}{thm}{Satz}
\newjointcountertheorem{lem}{thm}{Lemma}
\newjointcountertheorem{cor}{thm}{Corollary}
\newjointcountertheorem{prp}{thm}{Proposition}
\newjointcountertheorem{cnj}{thm}{Conjecture}
\newjointcountertheorem{que}{thm}{Question}
\newjointcountertheorem{fct}{thm}{Fact}
\theoremstyle{definition}
\newjointcountertheorem{dfn}{thm}{Definition}
\newjointcountertheorem{ntn}{thm}{Notation}
\newjointcountertheorem{rem}{thm}{Remark}
\newjointcountertheorem{nte}{thm}{Note}
\newjointcountertheorem{exl}{thm}{Example}

\def\Snospace~{\S{}}

\renewcommand{\mathbb}{\mathbf}
\newcommand{\flag}{\mathrm{Fl}}
\newcommand{\supp}{\mathrm{supp}}
\newcommand{\weakto}{\stackrel{\mathrm{w}}{\longrightarrow}}
\newcommand{\qbinom}[2]{\genfrac{[}{]}{0pt}{}{#1}{#2}}

\newcommand{\lra}{\longrightarrow}
\newcommand{\ra}{\rightarrow}

\newcommand{\vx}{\mathbf{x}}

\newcommand{\vj}{\mathbf{j}}

\newcommand{\vc}{\mathbf{c}}

\newcommand{\vf}{\mathbf{f}}

\newcommand{\vu}{\mathbf{u}}
\newcommand{\vp}{\mathbf{p}}
\newcommand{\vv}{\mathbf{v}}
\newcommand{\vw}{\mathbf{w}}

\newcommand{\vk}{\mathbf{k}}

\newcommand{\va}{\mathbf{a}}
\newcommand{\vb}{\mathbf{b}}
\newcommand{\vB}{\mathbf{B}}
\newcommand{\vBn}{\mathbf{B}^N}
\newcommand{\vL}{\mathbf{L}}
\newcommand{\vX}{\mathbf{X}}
\newcommand{\vY}{\mathbf{Y}}
\newcommand{\an}{a^N}
\newcommand{\bn}{b^N}
\newcommand{\jn}{j^N}
\newcommand{\lN}{l^N}
\newcommand{\Bn}{B^N}
\newcommand{\Cn}{C^N}
\newcommand{\en}{D^N}
\newcommand{\Ln}{L^N}
\newcommand{\Jn}{J^N}
\newcommand{\Tn}{T^N}
\newcommand{\Rn}{R^N}
\newcommand{\Yn}{Y^N}
\newcommand{\Xn}{X^N}
\newcommand{\vLn}{\mathbf{L}^N}
\newcommand{\vXn}{\mathbf{X}^N}
\newcommand{\vJ}{\mathbf{J}}
\newcommand{\vJn}{\mathbf{J}^N}
\newcommand{\vjn}{\mathbf{j}^N}

\newcommand{\vnu}{\mathbf{0}}

\newcommand{\Exp}{\mathbb{E}}
\newcommand{\Prob}{\mathbb{P}}
\newcommand{\pas}{\mathbb{P}-a.s.}
\newcommand{\lrad}{\stackrel{d}{\longrightarrow}}
\newcommand{\deq}{\stackrel{d}=}
\newcommand{\var}{{\mathrm{Var}}}
\newcommand{\cov}{{\mathrm{Cov}}}
\newcommand{\diag}{\mathrm{diag}}
\newcommand{\vvn}{\mathbf{v}^N}

\DeclareMathOperator{\Inv}{Inv}

\numberwithin{equation}{section}
\allowdisplaybreaks[4]							

\begin{document}

	\title[Distributions defined by $Q$-supernomials]{Distributions defined by $q$-supernomials, fusion products,
and Demazure modules}
	
	\author{Stavros Kousidis}
	\address{Rosenthalstr.~17, 53859 Niederkassel, Germany}
	\email{st.kousidis@googlemail.com}

	\author{Ernst Schulte-Geers}
	\address{Deutschherrenstr.~49, 53177 Bonn, Germany}
	\email{ernst.schulteg@t-online.de}
	
	\date{\today}

	\begin{abstract}
		We prove asymptotic normality of the distributions defined by $q$-supernomials, which implies asymptotic normality of the distributions given by the central string functions and the basic specialization of fusion modules of the current algebra of $\mathfrak{sl}_2$. The limit is taken over linearly scaled fusion powers of a fixed collection of irreducible representations. This includes as special instances all Demazure modules of the affine Kac-Moody algebra associated to $\mathfrak{sl}_2$. Along with an available complementary result on the asymptotic normality of the basic specialization of
graded tensors of the type $A$ standard representation, our result is a central limit theorem for a serious class of
graded tensors. It therefore serves as an indication towards universal behavior: The central string functions and
the basic specialization of fusion and, in particular, Demazure modules behave asymptotically normal, as the number of
fusions scale linearly in an asymptotic parameter, $N$ say.
	\end{abstract}

	\maketitle
	
	\section{Introduction}
	\label{sec:introduction}

	The $q$-supernomial coefficients encode certain integer partitions as polynomials in a variable $q$ of the form
	\begin{align*}
  		\sum_{j_1 + \cdots + j_m =a} 	q^{\sum_{i=1}^m j_i (j_i + \sum_{\ell = 1}^{i-1}L_\ell)} \prod_{\ell=1}^m \qbinom{L_{\ell}+j_{\ell+1}}{j_{\ell}}_q .
  	\end{align*}
	The $L_1,\ldots,L_m$ and $a$ are nonnegative integers and $\qbinom{a}{b}_q$ denotes the well known $q$-binomial coefficient that enumerates inversions in words. They were introduced by Schilling and Warnaar, who studied their symmetries, recurrences and $q$-series limits and gave a combinatorial interpretation as the enumeration of so-called $(L_1,\ldots,L_m)$-admissible generalized Durfee dissection partitions with exactly $a$ parts \cite{MR1665322}.

A natural motivation for the investigation of asymptotic statistical properties of the discrete distributions defined by
the coefficients of those polynomials is to understand the qualitative behavior of this special kind of integer
partitions for large values of the parameters $L_1,...,L_m \in \mathbb{Z}_+$.

However, our initial motivation for the study of $q$-supernomials is their
appearance as Hilbert series of fusion modules of the current algebra $\mathfrak{sl}_r \otimes \mathbf{C}[t]$ that were introduced by Feigin and Loktev \cite{MR1729359}, that is tensor products
of irreducible representations endowed with a grading that is encoded by the variable
$q$. The coefficients of those Hilbert series (also called string
functions) encode dimensions of certain isotropic components called weight spaces.
While the exact determination of those coefficients by coefficient extraction is certainly possible in any fixed
instance of a $q$-supernomial, the explicit description of those coefficients remains intractable and one usually
is satisfied with concrete expressions for their generating function, the $q$-supernomial.
For large values of the parameters $L_1,...,L_m$ one may expect that the
asymptotic behavior of the distributions defined by the $q$-supernomials
is governed by probabilistic limit theorems, so that precise
assertions about  ``typical'' behavior are possible. In this work we show
that limit theorems towards asymptotically normal behavior do indeed
hold, and deduce the following result for the (central) string functions of fusion
modules of the current algebra $\mathfrak{sl}_2 \otimes \mathbf{C}[t]$ and their so-called basic
specialization (a sum of string functions).

    {\itshape Consider a sequence of fusion modules $(\mathbf{C}^2)^{\ast \Ln_1 } \ast \cdots \ast (\mathbf{C}^{m+1})^{\ast \Ln_m}$ of the current algebra $\mathfrak{sl}_2 \otimes \mathbf{C}[t]$.
    Assume that the exponents grow on a linear scale, i.e.~$(\Ln_1 ,\ldots , \Ln_m)/N \ra {\bf a} \neq 0$ as $N\ra \infty$. Then, the central string functions and basic specialization of those modules behave asymptotically normal with mean and variance growing quadratically and cubically in $N$, respectively, with explicitly calculable leading terms.}

We give the precise statements describing the leading terms of this asymptotic expansion of the Hilbert series in
\autoref{thm:central-limit-basic-spec-fusion} and \autoref{thm:central-limit-strings-fusion}. Concerning the asymptotic growth of dominating weight spaces in fusion modules we conjecture that corresponding local limit theorems hold, and that similar results can be shown for the fusion of type $A$ symmetric power representations.
Furthermore, we specialize our findings to the case of Demazure modules that have been studied through different methods
earlier in the literature.

The material is broadly divided into two parts: the first part is devoted to the probabilistic-combinatorial problem of
deriving limit theorems for the q-supernomial distributions, and the second part explains the representation theoretic
interpretation of the limit theorems derived in the first part.

	\section{Distributions defined by $q$-supernomials}
	\label{sec:distributions}

		Let $\vL:=(L_1,\ldots,L_m) \in \mathbb{Z}_+^m$, $a \in \mathbb{Z}_+$, $j_{m+1}=0$.
		Consider the $q$-supernomial 
		\begin{align}
		  \label{eq:q-supernomial-modified}
		  \tilde{T}({\bf L},a)(q) & = \sum_{j_1 + \cdots + j_m =a} q^{\sum_{i=1}^m j_i (j_i + \sum_{\ell =
		      1}^{i-1}L_\ell)} \prod_{\ell=1}^m \qbinom{L_{\ell}+j_{\ell+1}}{j_{\ell}}_q ,
		\end{align}
		that enumerates $a$-restricted $\vL$-admissible partitions \cite{MR1665322}\footnote{The notation $\tilde{T}({\bf L},a)(q)$ is taken from there for consistency.}, and the cumulative generating function of the unrestricted number of $\vL$-admissible partitions
		\begin{align}
			\tilde{T}(\vL)(q):=\sum_{a=0}^{L_1+\cdots + L_m} \tilde{T}(\vL,a)(q) .
		\end{align}
We show that the unrestricted number and in certain
(typical) cases the $a$-restricted number of $\vL$-admissible partitions are asymptotically normally distributed with
asymptotic parameter being a convergent sequence $\tfrac 1N \vLn$, as $N \rightarrow \infty$.

	\subsection{Statistical notions}

		Standard sources are \cite{MR1324786,MR1700749,MR0228020,MR0270403}.
		All our random variables $X$ will be discrete and finite. Recall that the expected value of such a
random variable is the weighted average $\Exp(X) = \sum_{x} \Prob(X=x) x$. The covariance of two random variables $X$
and $Y$ is $\cov (X,Y) = \Exp( (X - \Exp(X))(Y - \Exp(Y)))$. They are said to be uncorrelated if $\cov (X,Y) = 0$. The
variance of $X$ is $\var(X) = \cov(X,X)$. Its probability generating function is $\Exp (q^X) = \sum_x \Prob(X=x) q^x $,
and the associated probability distribution $\mu_X = \sum_x \Prob(X=x) \delta_x$. Here, $\delta_x$ denotes the Dirac
distribution (point mass) at $x$. A sequence $X_N$ converges $\Prob$-almost surely ($\pas$ for short) to $X$ if $\Prob
(\lim_{N \rightarrow \infty} X_N = X ) =1$. Convergence and equality in distribution will be denoted by $\lrad$ and
$\deq$, respectively. $\mathcal{N}(\mu,\Sigma)$ will denote the normal distribution with mean $\mu$ and covariance
matrix $\Sigma$. Note that $\mathcal{N}(\mu,0) = \delta_\mu$. The conditional probability $\Prob (Y=y | X=x) =
\Prob(X=x)^{-1} \Prob(X = x, Y=y)$ is the probability of $Y$ taking the value $y$ given the occurence of the value $x$
for $X$. A mixture of distributions $\mu_{X_i}$ is a convex combination thereof, i.e.~$\sum_i w_i \mu_{X_i}$ for
some weights $w_i \geq 0$ with $\sum_i w_i =1$. The probability generating function of a mixture is $\sum_i w_i
\Exp(q^{X_i})$.

	\subsection{Preliminaries}
	\label{sec1}
		The distributions with probability generating function
		\begin{align}
			F_{a,b}(q):=\qbinom{a+b}{a}_q \big/ {a+b \choose a}
		\end{align}
were first investigated by Mann and Whitney \cite{MR0022058}, who showed:
		\begin{thm}\label{thm1}
			Let $\Inv_{a,b}$ be a random variable with distribution $F_{a,b}$. Then $\Inv_{a,b}$ has
expectation $\Exp(\Inv_{a,b})=	\frac{1}{2}ab$, variance $\var(\Inv_{a,b})=\frac{1}{12} ab(a+b+1)$, and
as $a,b \ra \infty$ one has
			\[
				\frac{\Inv_{a,b}-\Exp(\Inv_{a,b})}{\sqrt{\var(\Inv_{a,b})}}\lrad \mathcal{N}(0,1) .
			\]
		\end{thm}
		\noindent A corresponding local limit theorem was proved by Takacs \cite{MR845921}.
		
		\begin{rem}
It is well known that the $q$-binomial $\qbinom{a+b}{a}_q$ is the generating function for inversions in words of $a$ zeroes and $b$ ones. Consider a word (unordered sequence) $w=(w_1,\ldots,w_{n})$ of elements from an ordered set. A
$4$-tuple $(i,j,w_i,w_j)$ with $i<j$ and $w_i>w_j$ is then called an inversion. 
		\end{rem}
		
		Let us also recall the following classical result about the asymptotic normality of multinomial distributions
		(see e.g. \cite{mahoniannormality}).
		\begin{thm}
		\label{thm6}
			Let the sequence $\vBn$ have the multinomial distribution with parameters $N$ and $\vp=(p_0,p_1,\ldots,p_m)$. Then, we have mean $\Exp(\vBn)=N\vp$, covariance matrix $\cov(\vBn)=N\Sigma$, and
			\[
				\frac{\vBn-N\vp}{N^{1/2}} \lrad \mathcal{N}(\vnu,\Sigma) ,
			\]
			where $\Sigma=\diag(\vp)-\vp^t\vp$ (not of full rank).
		\end{thm}
		
		\subsubsection{Elementary definitions}
		\label{sec:elementary-definitions}
		
		We call a vector $\vj=(j_1,\ldots,j_m)\in \mathbb{Z}_+^m$ compatible to $\vL \in \mathbb{Z}_+^m$ if
$j_i\leq L_i+j_{i+1}$ for $i=1,\ldots m$, and a probability distribution $\vL$-compatible if the set of $\vL$-compatible
values has probability one. For $\vL$-compatible $\vj$ we let $\Inv(\vL,\vj)$ denote a random variable with probability
generating function
\begin{align}
  F(\vL,\vj)(q):=\prod_{i=1}^m F_{L_i+j_{i+1},j_i}(q) ,
  \end{align}
  where here and in the sequel $j_{m+1}=0$. We let
		\begin{align}
			Q(\vL,\vj)=\sum_{i=1}^m j_i \Big(j_i+\sum_{\ell=1}^{i-1}L_\ell \Big) .
		\end{align}
		We view the normalized generating function $\tilde{T}(\vL)(q) / \tilde{T}(\vL)(1)$ (and analogously $\tilde{T}(\vL,a)(q) / \tilde{T}(\vL,a)(1)$) as a mixture of probability generating functions
		
		\begin{align}
		  \tilde{T}(\vL)(q)/\tilde{T}(\vL)(1) = \sum_{\vj} q^{Q(\vL,\vj)} F (\vL,\vj)(q)\, \Prob(\vJ=\vj),
		\end{align}		
		weighted by $\Prob(\vJ=\vj)=\Prob(J_1=j_1,\ldots ,J_m=j_m)= \prod_{i=1}^m {L_i+j_{i+1} \choose j_i} / \tilde{T}(\vL)(1)$.
		The implicit dependency on the admission vector $\vL$ will be from now on suppressed in the notation for convenience. We furthermore define (the conditional distribution of) a random variable $Y$ by $\Prob(Y=i\;|\vJ=\vj)=\Prob(\Inv(\vL,\vj)=i)$, so that we can rephrase
		\begin{align}
		\label{eq:prob-gen-fctn-unrestricted-no-parts}
			 \tilde{T}(\vL)(q)/\tilde{T}(\vL)(1) = \Exp (q^{Q(\vL,\vJ) + Y}) .
		\end{align}
		Consequently, our interest lies in the distribution of the random variable
		\begin{align}
		\label{eq:T}
			T=Q(\vL,\vJ)+Y ,
		\end{align}
		whose asymptotics we will treat by splitting $T$ into an ``occupancy part'' $Q(\vL,\vJ)+\Exp(Y|\vJ)$, dependent only on  $\vJ$, and a remaining ``rest-inversion part'' $R=Y-\Exp(Y|\vJ)$ ``orthogonal'' to $\vJ$, and investigating the two parts individually.
		We note that by \autoref{thm1} we have
		\begin{align}
			\Exp(Y|\vJ)=\frac{1}{2}\left((L_m-J_m)J_m+\sum_{i=1}^{m-1} (L_i+J_{i+1}-J_i)J_i\right)=: e(\vL,\vJ)
		\end{align}
		so that $\Exp(Y|\vJ)$ is a quadratic function in $\vJ$. Let us call the distribution of $\vJ$ the \emph{mixing distribution} (for the total number of
$\vL$-admissible partitions). To explain the behavior of mixing distributions and to introduce a convenient way to
refer to them, let us describe a simple random experiment. We focus our attention on the unrestricted case first.
		
	\subsubsection{The probabilistic setup for the unrestricted case}
	\label{sec:experiment-unrestricted}
		
		Consider $m$ mutually independent random sources $\mathcal{S}_1,\mathcal{S}_2,\ldots,\mathcal{S}_m$
emitting words $W(1),\ldots, W(m)$. Each word  $W(i)=(X_1(i),X_2(i),\ldots,X_{L_i}(i))$ is a sequence of $L_i$ many
mutually independent letters $X_k(i)$, where each $X_k(i)$ is uniformly distributed from the alphabet $\{0,\ldots,i\}$. 
		For $0\leq i \leq k$ let $B_i(k):=\sum_{j=1}^{L_k} 1_{\{i\}}(X_j(k))$ denote the random variable
\emph{the number of appearances of letter $i$ in word $W(k)$}. Then, the random vector $\vB(k):=(B_0(k),\ldots,B_k(k))$
is the \emph{occupancy (statistic) of word $W(k)$}, $S(k):=\sum_{i=1}^{L_i} X_i(k)=\sum_{j=0}^k jB_j(k)$ is the
\emph{sum of word $W(k)$}, $\vB_\vL:=(\vB(1),\ldots,\vB(m))$ is the \emph{total occupancy of $W(1),\ldots,W(m)$}, and
$S_\vL:=\sum_{i=1}^m S(i)$ is called the \emph{total sum of words $W(1),\ldots,W(m)$}. Clearly, under the assumptions
above, $S_\vL$ is the sum of independent uniformly distributed random variables and we have
		\begin{align}
			\label{eq:mean-tensor}
			\Exp(S_\vL) & =\frac{1}{2}\sum_{i=1}^m iL_i , \\
			\label{eq:variance-tensor}
			\var(S_\vL) & = \frac{1}{12}\sum_{i=1}^m (i+2)i L_i.
		\end{align}
		The probability distribution of $\vJ$ arises in this experiment as follows.
		\begin{prp}
		\label{prop2}
			Let $\vB_\vL=(\vB(1),\ldots,\vB(m))$ be as above. That is, the random vectors $\vB(i)=(B_0(i),\ldots,B_i(i))$ are independent, where each $\vB(i)$ has a (uniform) multinomial distribution with parameters $L_i$ and $p_0=\ldots=p_i=\frac{1}{i+1}$. Define
			\[
				J_i:=\sum_{k=i}^m A_{k-i+1}(k) \mbox{, where } A_k(i):=\sum_{j=k}^i B_j(i).
			\]
			Then, the joint distribution of $(J_1,\ldots,J_m)$ is 
			\[
				\Prob(J_1=j_1,\ldots ,J_m=j_m)=\prod_{i=1}^m {L_i+j_{i+1} \choose
				j_i}/ (i+1)^{L_i} .
			\]
			where $\vj = (j_1,\ldots,j_m)\in \mathbb{N}_0^m$ (and with the usual convention about binomial
			coefficients that $\binom{a}{b} = 0$ unless $0 \leq b\leq a$). These numbers define a $\vL$-compatible probability distribution.
		\end{prp}
		\begin{proof}
			By formal generating functions it is clear that the joint generating function of $(J_1,\ldots,J_m)$ as defined above is
			\[
				\Exp (t_1^{J_1}\ldots t_m^{J_m})= \prod_{i=1}^m\big(\sum_{j=0}^i\;\prod_{k=i-j+1}^i t_k\big)^{L_i} / (i+1)^{L_i} .
			\]
			Now extract coefficients to see that this corresponds to the distribution defined above. The $\vL$-compatibility is obvious. 
		\end{proof}
		
		\begin{rem}
			The proof shows that  $\tilde{T}(\vL)(1)=\prod_{i=1}^m(i+1)^{L_i}$, which can be seen directly from the combinatorial definition of $q$-supernomials. Furthermore, since $A_{k-i+1}(k) = B_{k-i+1}(k) + B_{k-i+2}(k) + \cdots + B_{k}(k)$ counts the number of appearances of the highest $i$ letters in word $W(k)$, the $J_i$ may be described as the overall count of the $i$ highest non-zero letters in all words.
		\end{rem}

		\autoref{prop2} shows that the mixing distribution for the total number of $\vL$-admissible partitions 
may be realized as a simple linear transformation of $\vB_\vL$. We call $\vB_\vL$ the \emph{underlying occupancy
distribution}. This representation can be used for explicit calculations, and reduces the asymptotic treatment of $\vJ$
in the unrestricted case to the well known asymptotics of multinomial distributions as described in \autoref{thm6}.

		\subsubsection{The probabilistic setup for the $a$-restricted case}
		\label{sec:prob-setup-restricted}
			The same experiment as in \autoref{sec:experiment-unrestricted} describes the $a$-restricted case when we consider only the outcomes with total sum $S_\vL=a$. That is, for the $a$-restricted case the underlying occupancy distribution is $\vB_\vL\,|\,S_\vL=a$, i.e.~the distribution of $\vB_\vL$ conditioned by $S_\vL=a$. In order to have a succinct
wording for the ``most important'' restricted cases we make the following definition (Cf.~\eqref{eq:mean-tensor}).
			\begin{dfn}
			\label{def1} 
				We call the $a$-restricted cases with $a=\Exp(S_\vL)$ (resp.~$a=\Exp(S_\vL)\pm \tfrac 12$) \emph{central}. 
			\end{dfn}
			\noindent This terminology is justified due to the symmetry of the distribution of $S_\vL$ around $\Exp(S_\vL)$, and their central importance due to the strong law of large numbers, 
			\[
				\frac 1N \Exp(S_{\vLn}) \stackrel{\pas}{\xrightarrow{\hspace*{1cm}}} \frac 12 \sum_{k=1}^m ka_k \mbox{, as $\tfrac 1N \vLn \ra (a_1 , \ldots , a_m)$.}
			\]

	\subsection{General asymptotic considerations}
	\label{sec:general-asymptotics}
	
	Throughout this section let $\vLn$ be a sequence of admission vectors, and $\vJn$ a sequence of $\vLn$-compatible mixing distributions. We consider the sequence of inversion statistics 
	\[
		\Yn, \mbox{ defined by } \Prob(\Yn=i\;|\,\vJn=\vj)=\Prob(\Inv(\vLn,\vj)=i) ,
	\]
	and recall the associated definitions from \autoref{sec:elementary-definitions},
	\begin{align*}
		\Tn & = Q(\vLn,\vJn)+\Yn , \\
		\Rn & = \Yn - \Exp( \Yn | \vJn) , \\
		e(\vLn , \vJn) & = \Exp( \Yn | \vJn) .
	\end{align*}
	Let us note the asymptotic behavior of the conditional distribution of $\Rn$. 
	\begin{lem}\label{lem1}
		If $\tfrac 1N \vLn \ra \va \neq \vnu$, and $\vjn$ is a sequence of $\vLn$-compatible vectors such that $\tfrac 1N \vjn \ra \vb \neq \vnu$, then
		\[
			\frac{\Inv(\vLn,\vjn)-\Exp(\Inv(\vLn,\vjn)}{N^{3/2}}\lrad \mathcal{N}(0,v(\va,\vb))
		\]
		where $v(\va,\vb)=\frac{1}{12}\sum_{i=1}^m(a_i+b_{i+1}-b_{i})b_i(a_i+b_{i+1})$ (where $b_{m+1}=0$).
	\end{lem}
	
	\begin{proof}
		By definition, the random variable
		\[
			\Inv(\vLn,\vjn)-\Exp(\Inv(\vLn,\vjn))
		\]
		is distributed like the sum $\sum_{i=1}^m \Xn_i$ of $m$ independent random variables
		\[
			\Xn_i \deq \Inv(\Ln_i + \jn_{i+1} - \jn_i , \jn_i ) - \frac{1}{2}\left(\Ln_i + \jn_{i+1} - \jn_i )\,\jn_i \right) .
		\]
		Let $\an_i=\Ln_i+\jn_{i+1}-\jn_i$ and $\bn_i=\jn_i$. By assumption the limits $a_i=\lim \tfrac 1N \an_i$ and $b_i=\lim \tfrac 1N \bn_i$ exist.
		If $a_i=0$ or $b_i=0$, then $\tfrac{1}{N^{3/2}} \Xn_i \lrad 0\deq \mathcal{N}(0,0)$. If $a_i>0$
and $b_i>0$ we find that 
		$\tfrac{1}{N^3} \var(\Inv(\an_i,\bn_i) \ra \frac{1}{12} a_ib_i(a_i+b_i):=w(a_i,b_i)$, and \autoref{thm1} gives
		\[
			\frac{\Xn_i}{N^{3/2}}\lrad \mathcal{N}(0,w(a_i,b_i)) .
		\]
		Thus under the conditions above $\tfrac{1}{N^{3/2}} \sum_{i=1}^m \Xn_i \lrad \mathcal{N}(0,\sum_{i=1}^m w(a_i,b_i))=\mathcal{N}(0,v(\va,\vb))$.
	\end{proof}
	
	\begin{rem}
		The case $a_ib_i=0$ for all $i$ (that is $v(\va,\vb)=0$) is less interesting but not excluded. In this case we interpret  $\mathcal{N}(0,0):=\delta_0$ as the Dirac-measure (point mass) at $0$.
	\end{rem}

	Next we observe 
	that under mild conditions the limiting distributions of the (normalized) random variables $\Rn$ and $\vJn$ are asymptotically independent.
	
	\begin{thm}
	\label{thm2}
		Let $R(\va,\vb)$ denote a random variable with distribution $\mathcal{N}(0,v(\va,\vb))$.
		If $\tfrac 1N \vLn \ra \va \neq \vnu$, and if there exists $\vb\in\mathbb{R}^m_+$ and a positive semidefinite matrix $\Sigma\in\mathbb{R}^{m\times m}$ of positive rank such that
		\[
			\frac{\vJn- N\vb}{N^{1/2}} \lrad \mathcal{N}(\vnu,\Sigma) .
		\]
		Then, as $N\lra \infty$,
		\[
			\Big(\frac{\Rn}{N^{3/2}},\frac{\vJn -N\,\vb}{N^{1/2}}\Big)\lrad \big(R(\va,\vb),\mathcal{N}(\vnu,\Sigma)\big) ,
		\]
		where the constituents on the right-hand side are independent.
	\end{thm}
	
	\begin{proof}
		Let $A\subset \mathbb{R}$ be a Borel set and $f:\mathbb{R}^m \rightarrow \mathbb{R}$ be bounded and continuous. We have
		\begin{align*}
			& \Exp\,1_A\Big(\frac{\Rn}{N^{3/2}}\Big)\,f\Big(\frac{\vJn-N\vb}{N^{1/2}}\Big)\\
			& \quad \quad =\Exp\,\Prob\Big(\frac{\Rn}{N^{3/2}}\in A\;\big|\;\vJn \Big)\,f\Big(\frac{\vJn-N\vb}{N^{1/2}}\Big)\\
			& \quad \quad =\Exp\,\Prob\Big(\frac{\Inv(\vLn,\vJn)- e(\vLn,\vJn)}{N^{3/2}}\in A\Big)\,f\Big(\frac{\vJn-N\vb}{N^{1/2}}\Big) . 
		\end{align*}
		By Skorokhod's representation theorem (see \cite{MR1700749,MR0084897}) we may assume that $\tfrac{\vJn-N\vb}{N^{1/2}}\lrad\vX$, $\pas$, where $\vX\deq \mathcal{N}(\vnu,\Sigma)$. Then clearly $\tfrac 1N \vJn \ra \vb$, $\pas$, and by the preceding lemma

		\[
			\Prob\Big(\frac{\Inv(\vLn,\vJn)- e(\vLn,\vJn)}{N^{3/2}}\in A\Big)\lra \mathcal{N}(0,v(\va,\vb))\big(A\big), \pas 
		\]
		Therefore,
		\[
			\Exp\,1_A\Big(\frac{\Rn}{N^{3/2}}\Big)\,f\Big(\frac{\vJn-N\vb}{N^{3/2}}\Big)\lra \mathcal{N}(0,v(\va,\vb))\big(A\big)\;\Exp\big(f(\vX)\big) .
		\qedhere
		\]
	\end{proof}
	Finally we need a result that enables us to treat the occurring quadratic functions of $\vJn$. 
	For quadratic functions of asymptotically normal random vectors $\vXn$ one has the following elementary result.
	\begin{prp}
	\label{prop1}
		Assume there exists $\vb\in\mathbb{R}^m$ and a positive semidefinite matrix $\Sigma\in\mathbb{R}^{m\times m}$ of positive rank such that
		\[
			\frac{\Exp(\vXn)}{N}\lra \vb \;\;\;\;\mbox{ and }\;\;\; \frac{\vXn- N\vb}{N^{1/2}}\lrad \mathcal{N}(\vnu,\Sigma) .
		\]
		Let $M\in \mathbb{R}^{m\times m}$, $\vv \in \mathbb{R}^m$, and consider the quadratic function $q(\vx,\vv)=\vx^t M\vx+ \vv\vx^t$. If we assume that $\vvn$ is a sequence with $\tfrac 1N \vvn \ra \va$ and let $\vw	=\vb(M+M^t) +\va$, then
		\[
			\frac{ q\big(\vXn,\vvn \big)-q\big(\Exp(\vXn),\vvn \big)}{N^{3/2}}\lrad \mathcal{N}(0,\vw\Sigma\vw^t) .
		\]
		If additionally $\tfrac 1N \cov(\vXn) \lra \Sigma$, then
	  \[
	    \frac{\Exp(q(\vXn,\vvn))}{N^2}\lra q(\vb,\va) ,
	  \]
	  and
	  \[
	  \frac{\Exp\left(q(\vXn,\vvn	)-q(\Exp(\vXn),\vvn)\right)}{N} \lra \sum_{i,j} M_{i,j}\Sigma_{i,j} .
	  \]
		If furthermore $\Exp\left(\vXn_i-\Exp(\vXn_i)\right)^4/N^3\lra 0$ for all $i$, then 
	  \[
	    \frac{\var(q(\vXn,\vvn))}{N^3}\lra \vw\Sigma \vw^t .
	  \]
	\end{prp}
	
	Combining \autoref{prop1} with \autoref{thm2} gives a complete picture of the asymptotic distributions
	of the random variables considered above.	
	
	\begin{cor}
	\label{cor1}
		In the situation of \autoref{thm2} assume that additionally $\tfrac 1N \Exp(\vJn) \ra \vb$. Let
		$b_0 = b_{m+1}=0$, $\Jn_{m+1}=0$,
		$\Cn=\frac{1}{2} \sum_{i=1}^m
\big(\Ln_i + \Exp(\Jn_{i+1})-\Exp(\Jn_i)\big)\,\Exp(\Jn_i)$, 
and $\en=\sum_{i=1}^m \Exp(\Jn_i)\big(\Exp(\Jn_i)+\sum_{k=1}^{i-1}\Ln_{k}\big)$,
and let $\vc=\vc(\va,\vb)$ be the vector with
coordinates $c_i=a_i+b_{i+1}+b_{i-1}-2b_i$
and $\vf=\vf(\va,\vb)$ be the vector with coordinates $f_i=a_i+b_{i+1}+b_{i-1}+2b_i+2\sum_{\ell=1}^{i-1}a_\ell$.
 Then,
		\begin{align}
			\label{eq:upperline}
		 	\frac{e(\vLn,\vJn)- \Cn}{N^{3/2}} & \lrad \mathcal{N}(0,\tfrac{1}{4} \vc\Sigma \vc^t) , \\
			\frac{\Yn- \Cn}{N^{3/2}} & \lrad \mathcal{N}(0,\tfrac{1}{4} \vc\Sigma \vc^t + v(\va,\vb)) ,\\ 
			\frac{\Tn-(\Cn+\en)}{N^{3/2}}& \lrad \mathcal{N}(0,\tfrac{1}{4}\vf\Sigma \vf^t + v(\va,\vb)) .
		\end{align}
	\end{cor}

	\begin{proof}
		The first assertion follows directly from \autoref{prop1}. For the second assertion observe that by
\autoref{thm2} the limiting distribution is the convolution of the normal distributions $R(a,b)$ and the limiting
distribution in \eqref{eq:upperline}.
	\end{proof}

	Finally, the variance of the rest-inversion part $\Rn$ is as follows.
        \begin{lem}
        \label{lemvarr}
	  In the situation of \autoref{thm2} assume that additionally $\tfrac 1N \Exp(\vJn) \ra \vb$, and that
$\tfrac 1N \vJn$ is bounded. Then,
	  \[
	    \frac{\var(\Rn)}{N^3}\lra v(\va,\vb) .
	  \]
        \end{lem}
	\begin{proof}
	  We have $\Exp\big((\Rn)^2|\vJn\big) = \var(\Inv(\vLn,\vJn))$. Therefore
	  \[
	    \Exp\big((\Rn)^2|\vJn\big) =\tfrac{1}{12}\sum_{i=1}^m(\Ln_i +\Jn_{i+1} - \Jn_i ) \Jn_i (\Ln_i + \Jn_{ i+1}) ,
	  \]
	  and by our assumptions $\Exp\big((\Rn)^2|\vJn\big)/N^3$ converges boundedly to $v(\va,\vb)$. Hence
	  $\Exp\big((\Rn)^2\big)/N^3=\Exp\Big(\big((\Rn)^2|\vJn\big)\Big)/N^3\lra v(\va,\vb)$.
        \end{proof}
         
	\subsection{Unrestricted number of parts}
	\label{sec2}
		We first consider the total number $\tilde{T}(\vL)(q)$ of $\vL$-admissible partitions. In this case clearly $A_{k-i+1}(k)$ (as defined in \autoref{prop2}) has a binomial distribution with parameters $n=L_k$ and $p=\frac{i}{k+1}$, and hence each $J_i$ can be represented as a sum of independent binomial variables. Furthermore, the covariance of $A_{k-i+1}(k)$ and $A_{k-j+1}(k)$ can be computed as
		\[
			\cov(A_{k-i+1}(k),A_{k-j+1}(k))=\frac{L_k}{k+1}\left(\min(i,j)-\frac{ij}{k+1}\right) .
		\]
		We therefore have
		\begin{lem}
		\label{lem4}
			Consider $(J_1 , \ldots , J_m)$ as defined in \autoref{prop2}. Then,
			\begin{align}
				\Exp(J_i) & = i\sum_{k=i}^m\frac{L_k}{k+1} , \\
				\var(J_i) & = i\sum_{k=i}^m\frac{L_k}{k+1}-i^2\sum_{k=i}^m\frac{L_k}{(k+1)^2} , \\
\cov(J_i,J_j) & = \min(i,j)\sum_{k=\max(i,j)}^m\frac{L_k}{k+1}-ij\sum_{k=\max(i,j)}^m\frac{L_k}{(k+1)^2} .
			\end{align}
		\end{lem}
		
		Moreover, straightforward computations lead to the exact expectation value of the random variable
defined in \eqref{eq:prob-gen-fctn-unrestricted-no-parts}.

		\begin{lem}
		\label{lem5}
			Consider the random variables defined in \eqref{eq:prob-gen-fctn-unrestricted-no-parts}, and let
$s_i=\sum_{k=i}^{m} \frac{L_k}{k+1}$, $t_i=\sum_{k=1}^{i-1}L_k$. Then, for $\Exp (T) = \Exp(Y) + \Exp(Q(\vL,\vJ))$ we
have
			\begin{eqnarray*}
				\Exp(Y)&=&\frac{1}{2}\sum_{i=1}^m is_i^2-\frac{1}{4}\sum_{i=1}^m \frac{i}{i+1}L_i ,\\
				\Exp(Q(\vL,\vJ))&=&\sum_{i=1}^m i^2s_i^2 + \sum_{i=1}^m is_it_i+\sum_{i=1}^m\frac{i(i+2)}{6(i+1)}L_i .
			\end{eqnarray*}
		\end{lem}
		
		\begin{thm}
		\label{thm3}
			Let $\tfrac 1N \vLn \ra \va\neq \vnu$. Then, $\tfrac 1N \Exp(\Jn_i) \ra
i\sum_{k=i}^m \frac{a_k}{k+1}$ for each $i$ and 
			\[
				\frac{\vJn -\Exp(\vJn)}{N^{1/2}}\lrad \mathcal{N}(\vnu,\Sigma) ,
			\]
			where
			\[
\Sigma_{i,j}=\min(i,j)\sum_{k=\max(i,j)}^m\frac{a_k}{k+1}-ij\sum_{k=\max(i,j)}^m\frac{a_k}{(k+1)^2} .
			\]
		\end{thm}
		
		\begin{proof}
			The first assertion is clear. For the second assertion let $\vB_{\vLn}$ denote the underlying
total occupancy statistic. In \autoref{sec1} it was shown that the components $\vB(i)^N$ are indendent multinomial
distributions with parameters $\Ln_i$ and $\vu(i)$, where $u(i)_0=\ldots=u(i)_i=\frac{1}{i+1}$, and covariances
$\Ln_i \Sigma(i)$, $\Sigma(i)=\diag(\vu(i))-\vu^t(i)\vu(i)$. By \autoref{thm6}, and since $\tfrac 1N \Ln_i \ra
a_i$, we have
			\[
				\frac{\vB_{\vLn}-\Exp(\vB_{\vLn})}{N^{1/2}} \lrad \big(\mathcal{N}(\vnu,a_1\Sigma(1)),\ldots,\mathcal{N}(\vnu,a_m\Sigma(m))\big)
			\]
			where the components on the right-hand side are independent. Since $\vJn$ is a linear image of
$\vBn_\vL$ it is clear that $\vJn$ is asymptotically normal. The assertion about the covariance matrix is obvious.
		\end{proof}
		
		It is clear from \autoref{cor1} that under the conditions of \autoref{thm3} the number of admissible partitions is
		asymptotically normal, with expectation of order $N^2$ and variance of order $N^3$. The final condition
of \autoref{prop1} is fulfilled since the coordinates of the $\vB(i)$ are binomially distributed and therefore have
central fourth moments of order $L_i^2$. The boundedness condition of \autoref{lemvarr} is fulfilled since $\Jn_i\leq
\sum_{k=1}^{m} \Ln_k$.

	\subsection{The central restricted case}
	\label{sec:central-restricted}
	
	We consider the ''central region`` as discussed in \autoref{sec:prob-setup-restricted}.
	Let $s_N=\lfloor\Exp(S_{\vLn})\rfloor=\lfloor\tfrac 12 \sum_{i=1}^m iL_i\rfloor$. It is clear from the above
that the underlying occupancy distribution  $\vB_{\vLn}$ is the conditional distribution
	\[
	  \vB_{\vLn}=\big(\vY(1)^N,\ldots,\vY(m)^N\big)|\sum_{k=1}^m\sum_{i=0}^kY_i(k)^N=s_N
	  ,
	\]
	where $\vY(1)^N,\ldots,\vY(m)^N$ are independent random vectors, each $\vY(k)^N$ is multinomial with
parameters $\Ln_k$ and $p_0=\ldots=p_k=\tfrac{1}{k+1}$. This conditioning has the following effect on the
asymptotic distribution.

	\begin{thm}
	\label{thm7}
		Let $m>1$, $s_N=\lfloor \Exp(S_{\vLn})\rfloor$ and $\vu=(\vu(1),\vu(2), \ldots ,\vu(m))$ with
$\vu(k)=(\frac{1}{k+1},\ldots,\frac{1}{k+1})$. Assume that $\tfrac 1N \vLn \lra \va\not=\vnu$ and let
$\sigma^2(\va)=\tfrac{1}{12}\sum_{k=1}^m k(k+2)a_k$. Then,
		\[
		  \frac{\vBn-N\vu}{\sqrt{N}}\lrad \mathcal{N}(0,\Sigma) ,	
		\]
		where
		\begin{align}
		\label{eqsig}
		  \Sigma_{i,j}(k,\ell) & =a_k\delta_{k,\ell}\Big(\frac{1}{k+1}\delta_{i,j}-\frac{1}{(k+1)^2}\Big)  \\
		    \notag & \quad \quad -a_ka_\ell\frac{(k-2i)(\ell-2j)}{4(k+1)(\ell+1)\sigma^2(\va)} .
		\end{align}
	\end{thm}
	\begin{proof}
	  For ease of exposition we give the proof for the one component case $\vLn = (\Ln_1,\ldots,\Ln_{m-1}
	  ,\Ln_m) = (0,\ldots,0,N)$, the generalization is straightforward. Let $\vBn:=\vB(m)^N$, and
$X_1,\ldots,X_N$ be i.i.d.~random variables uniform on $\{0,\ldots,m\}$. Let $S_N:=\sum_{i=1}^N X_i$, $\mu=m/2$,
$\sigma^2=\tfrac{1}{12} m(m+2)$.
	  The probability generating function of $\vBn$ is
	  \[
	    \Exp(\prod_{i=0}^m t_i^{\Bn_i})=[x^{s_N}] (t_0+t_1x+\ldots + t_mx^m)^N/\Big((m+1)^n \Prob(S_N=s_N)\Big)
	    .
	  \]
          Hence the joint distribution is given by 
	  \[ 
	    \Prob(\Bn_0=k_0,\ldots,\Bn_m=k_m)={N\choose k_0,\ldots,k_m}/\big((m+1)^N\Prob(S_N=s_N)\big)
	  \]
	  with the constraints that $\sum_{i=0}^m k_i=N$ and $\sum_{i=1}^nik_i=s_N$. Since there are two linearly
independent linear constraints on the values of $\vBn$ we expect a $(m-1)$-dimensional limiting distribution. Let
$x_0,\ldots,x_m$ be real numbers with $\sum_{i=0}^m x_i=0$ and $\sum_{i=0}^m ix_i=0$, and let
$k_i=\tfrac{N}{i+1}+\sqrt{N}x_i$. By Stirling's approximation for the factorials for the numerator and the local limit
theorem for lattice distributions for the denominator we see
	  \[
	    (\sqrt{N})^{m-1}\Prob(\vBn=\vk)\lra \tfrac{1}{\sqrt{(2\pi)^{m-1}}}\sqrt{(m+1)^{m+1}\sigma^2 }
e^{-\tfrac{m+1}{2}(x_0^2+\ldots+x_m^2)}
	    .
	  \]
          A check that the expression on the right-hand side (considered as a function of $x_2,\ldots,x_m$, say) is
the marginal density of $(\mathcal{N}(0,\Sigma))_{2,\ldots,m}$ with $\Sigma$ as in \ref{eqsig} concludes the proof.
	\end{proof}
   
        For the convergence of moments we have
        \begin{prp}
        \label{prp:conv-moments-central-strings}
	  Under the conditions of \autoref{thm7} 
	  \[
	    \frac{\Exp(\vB(k)^N)}{N}\lra a_k\vu(k)\;\;\;\mbox{ and
}\;\;\frac{\cov(B_i(k)^N,B_j(\ell)^N)}{N}\lrad \Sigma_{i,j}(k,\ell)
	  \]
	  Furthermore, $\Exp\big(\vB_i(k)^N-\Exp(\vB_i(k)^N)\big)^4/N^3\lra 0$.
        \end{prp}
        
        \begin{proof}
             Again we restrict the exposition to the one component case and use the same notation as in the proof of
\autoref{thm7}. From the generating function given there we get
			\begin{align} 
			\Exp(\Bn_i)&=\frac{N}{m+1}\frac{\Prob(S_{N-1}=s_N-i)}{\Prob(S_N=s_N)}, \label{eqerw}\\
	\Exp(\Bn_i)^2&=\Exp(\Bn_i)+\frac{N(N-1)}{(m+1)^2}\frac{\Prob(S_{N-2}=s_N-2i)}{\Prob(S_N=s_N)}, \\
			\Exp(\Bn_i\Bn_j)&=\frac{N(N-1)}{(m+1)^2}\frac{\Prob(S_{N-2}=s_N-i-j)}{\Prob(S_N=s_N)} .
			\end{align}
                        By the local central limit theorem for lattice distributions \cite[Corollary VIII.3]{1165.05001}
we have
                           \[\Prob(S_N=k)=\frac{1}{\sqrt{2\pi
N}}e^{-\frac{(k-N\mu)^2}{2N\sigma^2}}\Big(1+O(N^{-1/2})\Big)\]
                        if $|\frac{(k-N\mu)^2}{2N\sigma^2}|<C$.
                   Applying this to the numerator and denominator shows
                        that the quotients $q_r(N):=\frac{\Prob(S_{N-r}=s_N-ri)}{\Prob(S_N=s_N)}$ are asymptotically of
the form
                        $q_r(N)=1-\frac{c_r}{2\sigma^2N}+O(N^{-3/2})$ where $c_r=r^2(m/2-i)^2$.
			Now, the asympotic assertion about the expectation follows immediately from the local limit
theorem, applied to numerator and denominator in \ref{eqerw}. The asymptotic assertion about the variance/covariance
follows from the formul\ae\;above using  the asymptotic form of $q_1(N),q_2(N)$. Concerning the asserted
convergence of the central fourth moment, note that the $r-th$ factorial moment of $\Bn_i$ is
                    \begin{align*}
		      & \Exp \big( \Bn_i(\Bn_i-1)\cdots (\Bn_i-r+1) \big)\\
		      & \quad \quad =\frac{N(N-1)\cdots
(N-r+1)}{(m+1)^r}\frac{\Prob(S_{N-r}=s_N-ri)}{\Prob(S_N=s_N)}
		      \end{align*}
                 After expressing the central fourth moment as a linear combination of factorial moments, plugging in
the asymptotic expressions for the $q_r(N)$, and noting that $c_4+6c_2-4c_3-4c_1=0$, one obtains
                           \[
                           		\Exp\big(\Bn_i-\Exp(\Bn_i)^4 \big)= N^3 (c_4+6c_2-4c_3-4c_1) +O(N^{5/2})=O(N^{5/2}) . 
		\qedhere
		\]
                 \end{proof}
		
		\begin{rem}
		\label{rem:asymptotic-moments-restricted}
		  Let $\tfrac 1N \vLn \ra \va\not=\vnu$.
		  A comparison to the unrestricted case, discussed in \autoref{sec2}, shows that asymptotically
the underlying total occupancy distributions are quite similar. They concentrate around the same expectations. In the
unrestricted case the components of the limiting distribution $\tfrac{1}{\sqrt{N}}(\vB_{\vLn}-\vu) \lrad Z$ are
independent normal vectors with
		  \[
		    \cov(Z(k))= a_k(\diag(\vu(k))-\vu(k)^t\vu(k))
		    .
		  \]
		  The components stay normal in the central restricted case, but the restriction causes an additional
negative correlation
		  \[
		    \cov(Z(k)_i,Z(\ell)_j)=-\frac{a_ka_\ell (2k-i)(2\ell-j)}{4 (k+1)(l+1)\sigma^2(\va)}
		  \]
		  between the components. This in turn forces the elements of the asymptotic covariance $\Sigma$ of
$\vJn$ to be smaller than in the unrestricted case, we compute
		  \[
		    \Sigma_{i,j,\mathrm{restricted}}=\Sigma_{i,j,\mathrm{unrestricted}}-\tfrac{ij}{2 \sigma^2(\va)}
c(i) c(j)
		  \]
		  with $c(i)=\sum_{k=i}^m\frac{k+1-i}{k+1} a_k$.
		\end{rem}

		Since $\vJn$ is a linear image of $\vBn$ its distribution is also asymptotically normal and it is clear
from \autoref{thm7} and \autoref{cor1} that $\Tn$ is asymptotically normal and the preceding results show that the
expectation resp. variance of $\Tn$ are  of $N^2$ resp.  $N^3$, but the variance in the restricted case will (on the
$N^3$ scale) be smaller than in the unrestricted case.

		\begin{rem}
	 		The cases $\tfrac 1N S_{\vLn} \ra a^\prime\not=a{:=\frac 12 \sum_{k=1}^m ka_k}$  can be treated by standard large deviation techniques. Again we may safely expect to find normality of the asymptotic distributions.
		\end{rem}

\section{Interpretation in terms of fusion and Demazure modules}

In the sequel we will consider characters of fusion modules of the current algebra defined by Feigin an Loktev in \cite{MR1729359}, and Demazure modules of the corresponding affine algebra. As mentioned in the introduction those characters can be interpreted as Hilbert series whose coefficients encode dimensions of so-called weight spaces. We will be mostly concerned with the current algebra $\mathfrak{sl}_2 \otimes \mathbf{C}[t]$ and the closely related affine Kac-Moody algebra $\widehat{\mathfrak{sl}}_2$. For further reading we refer the reader to \cite{MR2271991,MR1988973,MR2072647,MR2323538,MR2047177}, and for general facts about current and affine Kac-Moody algebras and their representation theory to \cite{MR2188930,MR1104219}.

\subsection{Fusion modules}
	Let us consider fusion modules for the current algebra $\mathfrak{sl}_2 \otimes \mathbf{C}[t]$ \cite{MR1988973,MR2072647}:
		\begin{align}
			\mathcal{F}(L_1,\ldots,L_m) = (\mathbf{C}^2)^{\ast L_1} \ast \cdots \ast (\mathbf{C}^{m+1})^{\ast L_m} .
		\end{align}
	    Feigin and Feigin prove in \cite[Theorem 5.1]{MR1988973} that, in terms of the $q$-supernomial $\tilde{T}({\bf L},a)(q)$ from \eqref{eq:q-supernomial-modified}, its graded character is given as
		\begin{align}
		\label{eq:character-fusion-T-supernomials}
			\chi (\mathcal{F}(L_1,\ldots,L_m))(z,q) = \sum_{a \in \mathbf{Z}} z^a \cdot \tilde{T}({\bf L},a)(q) .
		\end{align}
		\begin{figure}
		\label{plot of fusion}
		\begin{tikzpicture}[scale=0.7]
			\draw [->,gray] (-2.6,0) -- (2.5,0) node [right,black] {$z^a$};
			\draw [<-,gray] (-2.5,-17) node [below,black] {$q^i$} -- (-2.5,0.1);
			
			\draw [gray] (-2.6,-0.5) node [left,black] {$0$} -- (-2.4,-0.5);
			\draw [gray] (-2.6,-1) node [left,black] {$1$} -- (-2.4,-1);
			\draw [gray] (-2.6,-1.5) node [left,black] {$2$} -- (-2.4,-1.5);
			\draw [gray] (-2.6,-2) node [left,black] {$3$} -- (-2.4,-2);
			\draw [gray] (-2.6,-2.5) node [left,black] {$4$} -- (-2.4,-2.5);
			\draw [gray] (-2.6,-3) node [left,black] {$5$} -- (-2.4,-3);
			\draw [gray] (-2.6,-3.5) node [left,black] {$6$} -- (-2.4,-3.5);
			\draw [gray] (-2.6,-4) node [left,black] {$7$} -- (-2.4,-4);
			\draw [gray] (-2.6,-4.5) node [left,black] {$8$} -- (-2.4,-4.5);
			\draw [gray] (-2.6,-5) node [left,black] {$9$} -- (-2.4,-5);
			\draw [gray] (-2.6,-5.5) node [left,black] {$10$} -- (-2.4,-5.5);
			\draw [gray] (-2.6,-6) node [left,black] {$11$} -- (-2.4,-6);
			\draw [gray] (-2.6,-6.5) node [left,black] {$12$} -- (-2.4,-6.5);
			\draw [gray] (-2.6,-7) node [left,black] {$13$} -- (-2.4,-7);
			\draw [gray] (-2.6,-7.5) node [left,black] {$14$} -- (-2.4,-7.5);
			\draw [gray] (-2.6,-8) node [left,black] {$15$} -- (-2.4,-8);
			\draw [gray] (-2.6,-8.5) node [left,black] {$16$} -- (-2.4,-8.5);
			\draw [gray] (-2.6,-9) node [left,black] {$17$} -- (-2.4,-9);
			
			\draw [gray] (-2.6,-9.5) node [left,black] {$18$} -- (-2.4,-9.5);
			\draw [gray] (-2.6,-10) node [left,black] {$19$} -- (-2.4,-10);
			\draw [gray] (-2.6,-10.5) node [left,black] {$20$} -- (-2.4,-10.5);
			\draw [gray] (-2.6,-11) node [left,black] {$21$} -- (-2.4,-11);
			\draw [gray] (-2.6,-11.5) node [left,black] {$22$} -- (-2.4,-11.5);
			\draw [gray] (-2.6,-12) node [left,black] {$23$} -- (-2.4,-12);
			\draw [gray] (-2.6,-12.5) node [left,black] {$24$} -- (-2.4,-12.5);
			\draw [gray] (-2.6,-13) node [left,black] {$25$} -- (-2.4,-13);
			\draw [gray] (-2.6,-13.5) node [left,black] {$26$} -- (-2.4,-13.5);
			\draw [gray] (-2.6,-14) node [left,black] {$27$} -- (-2.4,-14);
			\draw [gray] (-2.6,-14.5) node [left,black] {$28$} -- (-2.4,-14.5);
			\draw [gray] (-2.6,-15) node [left,black] {$29$} -- (-2.4,-15);
			\draw [gray] (-2.6,-15.5) node [left,black] {$30$} -- (-2.4,-15.5);
			\draw [gray] (-2.6,-16) node [left,black] {$31$} -- (-2.4,-16);
			\draw [gray] (-2.6,-16.5) node [left,black] {$32$} -- (-2.4,-16.5);
			\draw [gray] (0,-0.1) node [above=4pt,black] {$4$} -- (0,0.1);
			\draw [gray] (0.5,-0.1) node [above=4pt,black] {$5$} -- (0.5,0.1);
			\draw [gray] (1,-0.1) node [above=4pt,black] {$6$} -- (1,0.1);
			\draw [gray] (1.5,-0.1) node [above=4pt,black] {$7$} -- (1.5,0.1);
			\draw [gray] (2,-0.1) node [above=4pt,black] {$8$} -- (2,0.1);
			\draw [gray] (-0.5,-0.1) node [above=4pt,black] {$3$} -- (-0.5,0.1);
			\draw [gray] (-1,-0.1) node [above=4pt,black] {$2$} -- (-1,0.1);
			\draw [gray] (-1.5,-0.1) node [above=4pt,black] {$1$} -- (-1.5,0.1);
			\draw [gray] (-2,-0.1) node [above=4pt,black] {$0$} -- (-2,0.1);

			\draw (0,-4.5) node {$1$};
			\draw (0,-5) node {$1$};
			\draw (0,-5.5) node {$3$};
			\draw (0,-6) node {$3$};
			\draw (0,-6.5) node {$4$};
			\draw (0,-7) node {$3$};
			\draw (0,-7.5) node {$2$};
			\draw (0,-8) node {$1$};
			\draw (0,-8.5) node {$1$};
	
			\draw (0.5,-7) node {$1$};
			\draw (0.5,-7.5) node {$2$};
			\draw (0.5,-8) node {$3$};
			\draw (0.5,-8.5) node {$3$};
			\draw (0.5,-9) node {$3$};
			\draw (0.5,-9.5) node {$2$};
			\draw (0.5,-10) node {$1$};
			\draw (0.5,-10.5) node {$1$};
	
			\draw (-0.5,-3) node {$1$};
			\draw (-0.5,-3.5) node {$2$};
			\draw (-0.5,-4) node {$3$};
			\draw (-0.5,-4.5) node {$3$};
			\draw (-0.5,-5) node {$3$};
			\draw (-0.5,-5.5) node {$2$};
			\draw (-0.5,-6) node {$1$};
			\draw (-0.5,-6.5) node {$1$};
	
			\draw (1,-9.5) node {$1$};
			\draw (1,-10) node {$1$};
			\draw (1,-10.5) node {$2$};
			\draw (1,-11) node {$2$};
			\draw (1,-11.5) node {$2$};
			\draw (1,-12) node {$1$};
			\draw (1,-12.5) node {$1$};
	
			\draw (-1,-1.5) node {$1$};
			\draw (-1,-2) node {$1$};
			\draw (-1,-2.5) node {$2$};
			\draw (-1,-3) node {$2$};
			\draw (-1,-3.5) node {$2$};
			\draw (-1,-4) node {$1$};
			\draw (-1,-4.5) node {$1$};
	
			\draw (1.5,-13) node {$1$};
			\draw (1.5,-13.5) node {$1$};
			\draw (1.5,-14) node {$1$};
			\draw (1.5,-14.5) node {$1$};
	
			\draw (-1.5,-1) node {$1$};
			\draw (-1.5,-1.5) node {$1$};
			\draw (-1.5,-2) node {$1$};
			\draw (-1.5,-2.5) node {$1$};
	
			\draw (2,-16.5) node {$1$};
	
			\draw (-2,-0.5) node {$1$};
		\end{tikzpicture}
		\caption{Plot of the character of $\mathcal{F}(0,4) = (\mathbf{C}^3)^{\ast 4}$, i.e.~$\chi(\mathcal{F}(0,4))(z,q) = \sum_{a \in \mathbf{Z}} z^a \cdot \sum_{j_1 + j_2 = a} q^{j_1^2 + j_2^2} \qbinom{4}{j_2}_q \qbinom{j_2}{j_1}_q$.}
		\end{figure}

Here $q$ refers to the grading as introduced by Feigin and Loktev \cite{MR1729359}.
	The specialization at $q=1$ of the graded character $\chi(\mathcal{F}(L_1, \ldots , L_m))(z,q)$
equals the character of the tensor product $(\mathbf{C}^2)^{\otimes L_1} \otimes
\cdots \otimes (\mathbf{C}^{m+1})^{\otimes L_m}$ of irreducible representations of $\mathfrak{sl}_2$, i.e.~the variable $z$ associates with the grading by the simple root $\alpha_1$ of $\mathfrak{sl}_2$. Since those characters multiply, i.e.~$\chi(V \otimes W) = \chi(V) \cdot \chi(W)$, the associated probability distributions convolute and asymptotic considerations reduce to the central limit theorem for sums of i.i.d.~random variables. Therefore, the specializations at $q=1$ are well understood from a statistical point of view and have been analyzed in great detail, e.g.~\cite{1062.22026}.	Much less studied is the so-called basic specialization\footnote{We borrow this terminology from Kac \cite[\S 1.5, 10.8, 12.2]{MR1104219} who analyzed this kind of specialization for characters of integrable highest weight modules $V(\Lambda)$, and obtained Macdonald's identities for Dedekind's $\eta$-function \cite[\S 12.2]{MR1104219}.} of those characters, i.e.~their evaluation at $z=1$. Now, we have
		\begin{thm}
		\label{thm:central-limit-basic-spec-fusion}
			Consider a sequence $(\mathbf{C}^2)^{\ast L_1^N} \ast \cdots \ast (\mathbf{C}^{m+1})^{\ast
L_m^N}$ of fusion modules of the current algebra $\mathfrak{sl}_2 \otimes \mathbf{C}[t]$. If $\tfrac 1N (L_1^N ,
\ldots , L_m^N) \ra {\bf a} \neq 0$, then the sequence of basic specializations, that is $\sum_{a \in \mathbf{Z}} \tilde{T}((L_1^N , \ldots , L_m^N),a)(q)$, behaves asymptotically normal with individual means
			\begin{align}
			\label{eq:mean-unr}
			  \mu_{(L_1^N , \ldots , L_m^N)} & = 
			    \sum_{i=1}^m \Bigg[ \Big( \frac i2 + i^2 \Big) \Big( \sum_{k=i}^{m} \frac{L_k^N}{k+1} \Big)^2
\\
			    \notag & \quad \quad + i \Big(  \sum_{k=i}^{m} \frac{L_k^N}{k+1} \Big) \Big(
\sum_{k=1}^{i-1}L_k^N \Big) + \Big(\frac{4i(i+2)-6i}{24(i+1)} \Big) L_i^N \Bigg],
			\end{align}
			and variance, as $N\ra\infty$,
			\begin{align}
			\label{eq:conv-var}
				\frac{1}{N^3} \sigma_{(L_1^N , \ldots , L_m^N)}^2 & \ra \tfrac{1}{4}\vf\Sigma \vf^t + v(\va,\vb).
			\end{align}
			Here, the vectors $\va$, $\vb$, $\vf$, the function $v$, and the matrix $\Sigma$ are given by
			\begin{align*}
				\va		& = (a_1,\ldots , a_m) = \lim_{N \ra \infty}\frac 1N (L_1^N , \ldots , L_m^N) , \\
				b_i		& = i \sum_{\ell = i}^m \frac{a_\ell}{\ell+1} , \\
				f_i		& = a_i+b_{i+1}+b_{i-1}+2b_i+2\sum_{\ell=1}^{i-1}a_\ell , \\
				v(\va,\vb)	& = \frac{1}{12}\sum_{i=1}^m(a_i+b_{i+1}-b_{i})b_i(a_i+b_{i+1}) \mbox{ (where $b_{m+1}=0$)} ,\\
				\Sigma_{i,j}	&
= \min(i,j)\sum_{k=\max(i,j)}^m\frac{a_k}{k+1}-ij\sum_{k=\max(i,j)}^m\frac{a_k}{(k+1)^2} .
			\end{align*}
		\end{thm}
		\begin{proof}
			Recall from \eqref{eq:prob-gen-fctn-unrestricted-no-parts} that the basic specialization is the generating function for the distribution of the random variable $\Tn=Q(\vLn,\vJn)+\Yn$ from \eqref{eq:T}. As such the mean $\mu_{(L_1^N , \ldots , L_m^N)}$ is equal to $\Exp(\Tn)$ as described in \autoref{lem5}. The assertions about the asymptotic normality and the variance are recollections from \autoref{thm2} and \autoref{cor1} which apply in the present situation due to \autoref{thm3}.
		\end{proof}
		
		The central string functions in fusion modules behave quite similar to the basic specialization. The following is just a recollection of the results in \autoref{sec:central-restricted}, and in particular the observations noted in \autoref{rem:asymptotic-moments-restricted}.

		\begin{thm}
		\label{thm:central-limit-strings-fusion}
		  Consider a sequence $(\mathbf{C}^2)^{\ast L_1^N } \ast \cdots \ast (\mathbf{C}^{m+1})^{\ast
L_m^N}$ of fusion modules of the current algebra $\mathfrak{sl}_2 \otimes \mathbf{C}[t]$. If $\tfrac 1N (L_1^N ,
\ldots , L_m^N) \ra {\bf a} \neq 0$, then the sequence of central string functions, that is $\tilde{T}((L_1^N , \ldots ,
L_m^N),s_N)(q)$ with $s_N = \lfloor\tfrac 12 \sum_{i=1}^m iL_i^N\rfloor$, behaves asymptotically normal with
asymptotic mean
			\begin{align}
			\label{eq:mean-re}
				& \frac{1}{N^2} \mu_{(L_1^N , \ldots , L_m^N),s_N} \ra \\
				 \notag & \quad \quad \sum_{i=1}^m \Bigg[ \Big( \frac i2 + i^2 \Big) \Big(
\sum_{k=i}^{m} \frac{a_k}{k+1} \Big)^2 + i \Big(  \sum_{k=i}^{m} \frac{a_k}{k+1} \Big) \Big(
\sum_{k=1}^{i-1}a_k \Big) \Bigg],
			\end{align}
			and asymptotic variance
			\begin{align}
			\label{eq:conv-var-central}
				\frac{1}{N^3} \sigma_{(L_1^N , \ldots , L_m^N), s_N}^2 & \ra
\tfrac{1}{4}\vf\Sigma \vf^t + v(\va,\vb).
			\end{align}
			Here, the vectors $\va$, $\vb$, $\vf$ and the function $v$ are as in \autoref{thm:central-limit-basic-spec-fusion}. The matrix $\Sigma$ is given as
			\begin{align*}
				\Sigma_{i,j}	& = -\frac{ij}{2 \sigma^2(\va)} c(i) c(j) +
\min(i,j)\sum_{k=\max(i,j)}^m\frac{a_k}{k+1}-ij\sum_{k=\max(i,j)}^m\frac{a_k}{(k+1)^2}
			\end{align*}
			where $\sigma^2(\va) = \tfrac{1}{12}\sum_{k=1}^m k(k+2)a_k$, and
$c(i) = \sum_{k=i}^m\frac{k+1-i}{k+1} a_k$.
		\end{thm}
	
	The following corresponding local central limit theorems should hold.
	
	\begin{cnj}
		\label{thm:loccentral-modules}
			In the notation of \autoref{thm:central-limit-basic-spec-fusion} let $X_{\mathcal{F}_N}$ denote a random variable with probability generating function the normalized basic specialization of the fusion module $\mathcal{F}_N = \mathcal{F}(\Ln_1 , \ldots , \Ln_m)$. Denote its mean $\mu_N = \mu_{(\Ln_1 , \ldots , \Ln_m)}$ and variance $\sigma^2_N = \sigma_{(\Ln_1 , \ldots , \Ln_m)}^2$. Then, uniformly in $k$ as $N \rightarrow \infty$,
			\begin{align}
			\label{eq:loc-central}
				\sqrt{2\pi} \sigma_N \cdot \Prob (X_{\mathcal{F}_N} = k) =
e^{-(k-\mu_N) / 2\sigma^2_N} + o(1) .
			\end{align}
			Here, $\sigma^2_N$ can be replaced by the explicit expression $N^3 (\tfrac{1}{4}\vf\Sigma
\vf^t + v(\va,\vb))$ from \eqref{eq:conv-var}.
			In particular, the dimension of the $\mathfrak{sl}_2$ submodule in $\mathcal{F}_N$ of degree $k$ grows as \eqref{eq:loc-central}.
		\end{cnj}
		
		\begin{cnj}
		\label{thm:loccentral-weights}
			In the notation of \autoref{thm:central-limit-strings-fusion} let $S_N$ be a random variable with
probability generating function the normalized central string function
			\[
			  \Exp (q^{S_N}) = q^{-\tfrac 12 \lN_1 \lN_m} \tilde{T}(\vLn,s_N)
			  ,
			\]
			where $\lN_1 = \sum \Ln_i$ and $\lN_m = \sum i \Ln_i$.
			Let
			\begin{align*}
			  \mu & = \lim_{N \ra \infty} \tfrac{1}{N^2} \Big( \mu_{(\Ln_1, \ldots , \Ln_m),s_N} -
\tfrac 12 \lN_1 \lN_m \Big) , \\
			  \sigma^2 & = \lim_{N \ra \infty} \tfrac{1}{N^3} \Big(\sigma^2_{(\Ln_1, \ldots , \Ln_m),s_N} \Big) .
			\end{align*}
		Then, uniformly in $k$ as $N \ra \infty$,
			\begin{align}
			\label{eq:loc-central-restricted}
				\sqrt{2\pi} \sigma \cdot \Prob (S_N = k) = e^{-(k-\mu) / 2\sigma^2} +
o(1)
			\end{align}
			In particular, the dimension of the weight space with coordinates $\tfrac 12 \lN_m\alpha_1$ and $-k\delta$ grows as \eqref{eq:loc-central-restricted}.
		\end{cnj}
		
		Complemented by a result on the asymptotic normality of the basic specialization of
graded tensors of the type $A$ standard representation (see \cite{bk11flags}) we have a central limit theorem for a serious class of graded tensors. We will conclude with a general conjecture on the fusion modules of symmetric power representations for the current algebra $\mathfrak{sl}_r \otimes \mathbf{C}[t]$ in \autoref{sec:fusion-symmetric-power}.

	\subsection{Demazure modules}	
	
	It is well-known that Demazure modules $V_w(\Lambda)$ associated to $\widehat{\mathfrak{sl}}_2$ carry a $\mathfrak{sl}_2 \otimes \mathbf{C}[t]$-module structure and as such are special instances of fusion modules (see e.g. \cite[\S 1.5.1]{MR2271991} or \cite[\S 3.5]{MR2323538}). To be precise, there are isomorphisms of $\mathfrak{sl}_2 \otimes \mathbf{C}[t]$-modules
		\begin{align}
		\label{eq:graded-factorization-phenomenon}
			V_w (m\Lambda_0 + n\Lambda_1) \cong \begin{cases} \mathcal{F}({\bf 0},\underset{m+1}{1},{\bf 0},\underset{m+n+1}{l(w)-1}) & , w = w' s_0 , \\ \mathcal{F}({\bf 0},\underset{n+1}{1},{\bf 0},\underset{m+n+1}{l(w)-1}) & , w = w' s_1 .\end{cases}
		\end{align}
		Here, we write $l(w)$ for the length of a reduced decomposition of $w \in W^{\mathrm{aff}}$ in the affine Weyl group of $\widehat{\mathfrak{sl}}_2$. Note that the elements of $W^{\mathrm{aff}}$ can be expressed as the following products of the simple reflections $s_0$ (corresponding to the imaginary root) and $s_1$ (corresponding to the simple root $\alpha_1$):
		\[
			(s_1s_0)^N s_1 , N \geq 0 \mbox{ and } (s_0 s_1)^N , (s_1s_0)^N , s_0(s_1s_0)^{N-1}, N > 0.
		\]
  		The characters of Demazure modules can be identified as special instances of fusion modules through a series of translations (multiplication by $z^a$), reflections (evaluation at the reciprocal $1/q$), and rotations (evaluation at mixed monomials $z q^i$), respectively.
		\begin{prp}
		\label{characters-demazure-through-fusion}
			Consider the Demazure module $V_w = V_w(m\Lambda_0 + n\Lambda_1)$ of $\widehat{\mathfrak{sl}}_2$ of fixed highest weight $\Lambda = m\Lambda_0 + n\Lambda_1$ and recall the character formula for fusion modules of the current algebra $\mathfrak{sl}_2 \otimes \mathbf{C}[t]$ \eqref{eq:character-fusion-T-supernomials}. The character $\chi(V_w)(z,q)$, written in the coordinates $z=e^{-\alpha_1}$ and $q = e^{-\delta}$ ($\alpha_1$ is the simple root of $\mathfrak{sl}_2$, $\delta$ denotes the imaginary root), is given by
			\begin{eqnarray}
				&& \mbox{For the trivial element $w = \mathbf{1}$ one has $\chi(V_\mathbf{1})(z,q) = e^\Lambda$.} \\
				\label{eq:character-demazure-fusion-1}
				&& \mbox{For $w = (s_1s_0)^N s_1, N \geq 0$ one has}\\
				&& \notag \quad \mbox{$e^{-\Lambda}\chi(V_w)(z,q) = $}\\
				&& \notag \quad \quad \quad \mbox{$z^{-(n+m)N - n/2} q^{N^2 m + N(N+1)n} \chi(\mathcal{F}(0,{\bf L}_{w})(zq^{-2N-1},q)$}\\
				&& \notag \mbox{where ${\bf L}_{w} = (L_1 , \ldots , L_{m+n}) = ({\bf 0},\underset{n}{1},{\bf 0},\underset{m+n}{2N})$.}\\
				\label{eq:character-demazure-fusion-2}
				&& \mbox{For $w = (s_0s_1)^N, N \geq 0$ one has}\\
				&& \notag \quad \mbox{$e^{-\Lambda}\chi(V_w)(z,q) = $}\\
				&& \notag \quad \quad \quad \mbox{$z^{-(n+m)N - n/2} q^{N^2 m + N(N+1)n} \chi(\mathcal{F}(0,{\bf L}_{w})(zq^{-2N-1},q)$}\\
				&& \notag \mbox{where ${\bf L}_{w} = (L_1 , \ldots , L_{m+n}) = ({\bf 0},\underset{n}{1},{\bf 0},\underset{m+n}{2N-1})$.}\\
				\label{eq:character-demazure-fusion-3}
				&& \mbox{For $w = (s_1s_0)^N, N > 0$ one has}\\
				&& \notag \quad \mbox{$e^{-\Lambda}\chi(V_w)(z,q) = $}\\
				&& \notag \quad \quad \quad \mbox{$z^{-(n+m)N + n/2} q^{N^2 m + N(N-1)n} \chi(\mathcal{F}(0,{\bf L}_{w})(zq^{-2N},q)$}\\
				&& \notag \mbox{where ${\bf L}_{w} = (L_1 , \ldots , L_{m+n}) = ({\bf 0},\underset{m}{1},{\bf 0},\underset{m+n}{2N-1})$.}\\
				\label{eq:character-demazure-fusion-4}
				&& \mbox{For $w = s_0(s_1s_0)^{N-1}, N > 0$ one has}\\
				&& \notag \quad \mbox{$e^{-\Lambda}\chi(V_w)(z,q) = $}\\
				&& \notag \quad \quad \quad \mbox{$z^{-(n+m)N + n/2} q^{N^2 m + N(N-1)n} \chi(\mathcal{F}(0,{\bf L}_{w})(zq^{-2N},q)$}\\
				&& \notag \mbox{where ${\bf L}_{w} = (L_1 , \ldots , L_{m+n}) = ({\bf 0},\underset{m}{1},{\bf 0},\underset{m+n}{2N-2})$.}
			\end{eqnarray}
			The sum of the entries in ${\bf L}_{w}$ represents the length $l(w)$ of the Weyl group element $w$. When either $n$ or $m$ equals $0$, then ${\bf L}_{w} = ({\bf 0},l(w))$.
		\end{prp}

		\begin{proof}
			Feigin \cite[(11)]{MR2072647} denotes an integrable highest weight representation $L_{i,k} =
\mathcal{U}(\widehat{\mathfrak{sl}}_2).v_{i,k}$ with highest weight vector $v_{i,k}$ such that $c.v_{i,k} = kv_{i,k}$,
$h_0 . v_{i,k} = i v_{i,k}$, and $d.v_{i,k} = 0$. In our notation, the canonical central element is $c = \alpha_0^\vee +
\alpha_1^\vee$, the coroot is $h_0 = \alpha_0^\vee$, and the scaling element $d$ is given by $\alpha_0(d) =1$ and
$\alpha_1(d)=0$. Therefore, by comparison of the highest weight vector we have $L_{i,k} = V( (k-i)\Lambda_0 + i\Lambda_1
)$. The bigrading is chosen according to the action of $h_0$ and $d$, and consequently, the character is denoted in the
monomials $e^{\alpha_0} = e^{\delta - \alpha_1}$ and $e^{-\delta}$, respectively. By \cite[Corollary 3.1]{MR2072647}
each such module $L_{i,k}$ can be constructed as an inductive limit of fusion products, that is $L_{i,k} =
\mathbf{C}^{i+1} \ast (\mathbf{C}^{k+1})^{2\infty}$. Each fusion product can be identified with the corresponding
Demazure module $V_w( (k-i)\Lambda_0 + i\Lambda_1 )$ by comparing the weights of the extremal weight vectors described in \cite[\S 1]{MR2072647}. Now apply the character formula \cite[Theorem 5.1]{MR1988973}, noting that $e^{\alpha_0} = zq^{-1}$.
		\end{proof}
		
		We can now compare our findings to established results in the literature \cite{bk11flags,bk10law,MR2852488,janson}.

	\subsubsection{The unrestricted one component case}

		Consider $\vL =(0,\ldots,0,N) \in \mathbf{Z}_+^m$ in which case the distribution of $\Yn$ has been investigated as a generalized Galois numbers by Bliem and Kousidis \cite{bk11flags} and later on by Janson \cite{janson}. They studied the random variables $\Yn$ with probability generating function
		\begin{align}
		\label{eq:galois}
			\Exp(q^{\Yn})=\frac{1}{(m+1)^N}\sum_{\substack{ (k_0,\ldots,k_m)\in \mathbb{N}_0^{m+1}\\ k_0+\cdots +k_m=N}} \qbinom{N}{k_0,\ldots,k_m}_q	
		\end{align}
		Note that if we let $B_0=L_m-J_m,B_1=J_m-J_{m-1},\ldots,B_{m-1}=J_2-J_1,B_m=J_1$ it becomes evident that this distribution coincides with the distribution	of $\Yn$ in the one component case. Now, Bliem and Kousidis showed
		\begin{thm}[\protect{\cite[Theorem 3.5]{bk11flags}}]
		\label{thm4}
			Consider the random variables $\Yn$ defined through \eqref{eq:galois}. Then,
			\begin{align*}
				\Exp(\Yn)&=\frac{1}{4}\frac{m}{m+1} N(N-1) , \\
				\var(\Yn)&=\frac{1}{72}\frac{(m+1)^2-1}{(m+1)^2} N(N-1)(2N+5) , \\
				\mbox{ and }\;\;\;
				\frac{\Yn-\Exp(\Yn)}{N^{3/2}}&\lrad \mathcal{N}\big(0,\frac{1}{36}\frac{(m+1)^2-1}{(m+1)^2}\big) .
			\end{align*}
		\end{thm}
		Janson derived the same result in a variety of ways, proved a corresponding local central limit
theorem, and gave different interpretations of the distribution of $\Yn$. He also
showed joint convergence of $\Yn$ and $\vBn$:
		\begin{thm}[\protect{\cite[Theorem 2.4]{janson}}]
		\label{thm5}
			Let $\vBn$, $\Yn$ be as above. Then,
			\begin{align}
				\label{eq1}
				& \Big(\frac{\Yn-\Exp(\Yn)}{N^{3/2}},\frac{\vBn-\Exp(\vBn)}{N^{1/2}}\Big) \\
				\notag
				& \quad \quad \quad \lrad\Big(\mathcal{N}\big(0,\frac{1}{36}\frac{(m+1)^2-1}{(m+1)^2}\big),\mathcal{N}(\vnu,\Sigma)\Big) ,
			\end{align}
			where the constituents on the right hand side are independent, and the matrix $\Sigma$ is given by $\Sigma_{i,j}=\frac{1}{m+1}(\delta_{i,j}-\frac{1}{m+1})$.
		\end{thm}

		Let us compare these results to our findings above. We obtain from \autoref{lem5}:
		\[
			\Exp(\Yn)=\frac{1}{4}\frac{m}{m+1}\,N^2 -\frac{1}{4}\frac{m}{m+1}N ,
		\]
		which agrees with the expectation given in \autoref{thm4}. Further, we have
		\[		
			\Exp(\Jn_i)=i\frac {N}{m+1}, \;\;\;b_i=\frac{i}{m+1},\;\;\;  a_0=\ldots=a_{m-1}=0,\,a_m=1 ,
		\]
		and find that
		\[
			\Cn=\frac{m(m+1)}{4}\,N^2 ,
		\]
		and
		\[
			v(\va,\vb)=\frac{1}{12(m+1)^3}\sum_{i=1}^{m}i(i+1)=\frac{1}{36}\frac{m(m+2)}{(m+1)^2} .
		\]
		Finally, $\vc=\vnu$, where $\vc=\vc(\va,\vb)$ is as in \autoref{cor1}. Thus, by \autoref{cor1} we have
		\[		
			\frac{\Yn-\Cn}{N^{3/2}}\lrad\mathcal{N}\big(0,\frac{1}{36}\frac{m(m+2)}{(m+1)^2}\big) ,
		\]
		which is equivalent to the weak convergence assertion in \autoref{thm4}, and hence establishes an independent proof. Moreover, by \autoref{cor1}
		\[
			\frac{e(\vLn,\vJn)-\Cn}{N^{3/2}}\lrad\mathcal{N}(0,0)=\delta_0 ,
		\]
		that together with \autoref{thm1} independently proves Janson's \autoref{thm5}.

		Note that \autoref{thm5} as it stands does not generalize to more general distributions.  As an example let $\vBn$ be multinomial with parameters $N,\vp$ where $\vp$ is not uniform. Here we get from \autoref{cor1} that 
		\[
			\frac{e(\vLn,\vJn)-\Cn}{N^{3/2}}\lrad\mathcal{N}(0,v_1(\vp)) ,
		\]
		where $v_1(\vp)=\frac{1}{4}\big(\sum_{i=0}^m p_i^3 -(\sum_{i=0}^mp_i^2)^2\big)$. The corresponding joint limiting distribution (on the right hand side of \eqref{eq1}) is  normal, but the constituents are not independent.
		
	\subsubsection{The unrestricted two component case}
	\label{sec:comparison-demazure}
		
		For the two component case of $\vL = (L_1,\ldots,L_m)$, i.e.~$L_m=M$, $L_k=K$ for a $k<m$, and all other $L_i = 0$, we find
		\begin{align*}
			\Exp(T_\vL)&= \frac{1}{12}\frac{m(4m+5)}{m+1}\,M^2 +\frac{1}{12}\frac{m(2m+1)}{m+1}\,M
					+\frac{1}{12}\frac{k(4k+5)}{k+1}\,K^2 \\
				&  \quad +\frac{1}{12}\frac{k(2k+1)}{k+1}\,K +\frac{1}{2} m\, KM + \frac{1}{6}\frac{k(k+2)}{m+1}\,KM .
		\end{align*}
		For $K=1$ this simplifies to
		\begin{align*}
			\Exp(T_\vL)&= \frac{1}{12}\frac{m(4m+5)}{m+1}\,M^2 +\frac{1}{12}\frac{m(8m+7)+2k(k+2)}{m+1}\,M+\frac{k}{2} .
		\end{align*}
		Let us interpret this in terms of the Demazure modules $V_w(m \Lambda_0 + n \Lambda_1)$ from \eqref{eq:graded-factorization-phenomenon}. The random variables $X_w$ having probability generating function the basic specialization of the character $\chi(V_w(\Lambda))$ are given due to \autoref{characters-demazure-through-fusion} by translations and rotations (averaging over the random variable $S_{\vL_w}$) as follows
		\begin{align}
			& \mbox{Equations \eqref{eq:character-demazure-fusion-1} and \eqref{eq:character-demazure-fusion-2} read as:}\\
			\notag & \quad \quad \mbox{$X_w = N^2m +N(N+1)n + (-2 N-1)\cdot S_{\vL_w} + T_{\vL_w}$.}\\
			\label{eq:random-trans-rot-2}
			& \mbox{Equations \eqref{eq:character-demazure-fusion-3} and \eqref{eq:character-demazure-fusion-4} read as:}\\
			\notag & \quad \quad \mbox{$X_w = N^2m +N(N-1)n - 2 N \cdot S_{\vL_w} + T_{\vL_w}$.}
		\end{align}
		The cases covered here correspond to the cases found in \cite[Theorem 4.1]{MR2852488}. Let us
restrict for simplicity reasons to \eqref{eq:random-trans-rot-2} for $w=(s_1 s_0)^N$, and compare our findings to
\cite[Theorem 4.1]{MR2852488}, where the corresponding case is \cite[(4.1)]{MR2852488} for even $N$.
		\begin{thm}[\protect{\cite[(4.1) in Theorem 4.1]{MR2852488}}]
			For $\vL$ with entries $0$ except $L_m=1$, $L_{m+n}=2N-1$, and with $U=2N-1$, $u=m+n$ one has
			\begin{align}	
				\label{bk-result} & \Exp(N^2m+N(N-1)n-2S_\vL N+T_\vL)  \\
				\notag & \quad \quad =\frac{2Um(m+2)+U(U-1)u(u+2)}{12(u+1)}+\frac{U-1}{2}\frac{u}{2}+\frac{m}{2} .
			\end{align}
		\end{thm}
		\begin{proof}
		We can establish \eqref{bk-result} by the computation of the left-hand side through the linearity of
$\Exp (.)$ and the mean of the random variables $S_\vL$ and $T_\vL$. That is
			\[		
				\Exp(T_\vL)=\frac{1}{12}\frac{u(4u+5)}{u+1}\,U^2 +\frac{1}{12}\frac{u(8u+7)+2m(m+2)}{u+1}\,U+\frac{m}{2} ,
			\]
			and $\Exp(S_\vL)=\frac{1}{2} (mL_m+uL_u)=\frac{1}{2}(m+uU)=\frac{1}{2}(2uN-n)$ which gives
			\[		
				\Exp(N^2m+N(N-1)n-2S_\vL N)=-N^2u=-\frac{(U+1)^2}{4}u .
				\qedhere
			\]
		\end{proof}

		Finally, we settle a question that was posed by Bliem and Kousidis in \cite{bk10law}.
		\begin{lem}[\protect{Cf.~\cite[Conjecture 8.3]{bk10law}}]
		\label{asymptotic-concentration-demazure}
			Fix a dominant integral weight $\Lambda$ and a sequence $(w_N)$ in $W^{\mathrm{aff}}$ such that $l(w_N) \to \infty$. Let $\mu_N$ be the joint distribution of the degree and the finite weight in $V_{w_N}(\Lambda)$. Let $\tilde\mu_N$ be the distribution obtained from $\mu_N$ by normalizing to a probability distribution and rescaling the two coordinates individually so that $\supp(\tilde\mu_N)$ just fits into the rectangle $[0,1] \times [-1,1]$.
		Then, as $N \to \infty$,
		\[
			\tilde\mu_N \weakto \delta_{\left( \frac{\langle c, \Lambda \rangle +2}{3(\langle c, \Lambda \rangle+1)}, 0 \right)} ,
		\]
		where $c = \alpha_0^\vee + \alpha_1^\vee$ denotes the canonical central element.
		\end{lem}

		\begin{proof}
			We consider only the Demazure modules $V_{(s_1 s_0)^N}(m\Lambda_0+n\Lambda_1)$ as the other
cases can be derived similarly. Let $r_N$ denote the maximal degree in these Demazure modules, i.e.~$r_N =
N^2m+N(N-1)n$. Let $u=m+n = \langle c, \Lambda \rangle$ denote the level of the representation, and consider the random
variable with probability generating function given by the basic specialization of our Demazure module, that is
			\begin{align*}
				X_N & = r_N - 2N S_{\vLn} + \Tn \\
					& = \Exp\big(r_N - 2N \Exp(S_{\vLn})+\Tn\big) -2N(S_{\vLn}-\Exp(S_{\vLn}))+(\Tn-\Exp(\Tn))
			\end{align*}
			The probability distribution of $X_N$ and $\tfrac{1}{r_N} X_N$ is the first coordinate of $\mu_N$ and $\tilde\mu_N$ for the Weyl group element $w_N = (s_1 s_0)^N$, respectively. Now, equivalent to the asserted weak convergence of $\tilde\mu_N$ we have
			\begin{align*}
				\frac{X_N}{r_N} \lrad \frac{u+2}{3(u+1)} ,
			\end{align*}
			since
			\begin{align*}
				\frac{\Exp\big(r_N - 2N S_{\vLn}+\Tn\big)}{r_N} & \ra \frac{1}{3}\frac{4u+5}{u+1}-1=\frac{u+2}{3(u+1)} ,
			\end{align*}
			and by \eqref{eq:variance-tensor} and \autoref{cor1} we have the convergences in distribution
			\begin{align*}
				\frac{S_{\vLn}-\Exp(S_{\vLn})}{N} & \lrad 0 , \quad \mbox{ and } \quad  \frac{\Tn-\Exp(\Tn)}{N^2} \lrad 0 .
			\end{align*}
			Since it is well known that the second coordinate of $\tilde\mu_N$ concentrates in $0$, the claim follows.
		\end{proof}
	
	\section{Fusion of symmetric power representations}
	\label{sec:fusion-symmetric-power}

		The Kostka numbers are the coefficients in the expansion
		\begin{align}
		\label{eq:schur-monomial}
			 \prod_i h_{\xi_i}({\bf x}) = \sum_{\xi} K_{\eta,\xi} \cdot s_{\eta}({\bf x}) 
		\end{align}
		of the product of complete symmetric functions $h_{\xi_i}$ in terms of the Schur functions $s_{\eta}$.
		The Kostka polynomials $K_{\eta,\mu}(q)$ generalize the Kostka numbers in the sense that $K_{\eta,\mu}(1) = K_{\eta,\mu}$. They give the transition matrix between the Schur function $s_{\eta}$ and Hall-Littlewood function $P_{\mu}$, i.e.
		\begin{align}
		\label{eq:eq:halllittlewood-schur}
			s_{\eta}({\bf x}) = \sum_{\mu} K_{\eta,\mu}(q) \cdot P_{\mu}({\bf x},q)   .
		\end{align}
		A standard reference for the above functions is \cite{0824.05059}.

		Now, the $q$-supernomial $S_{\xi,\mu}(q)$ \cite{MR1663325,MR1768934,MR1903985,MR1665322,MR1690046} is defined as the combination of \eqref{eq:schur-monomial} and \eqref{eq:eq:halllittlewood-schur}, i.e.~as the transition between the above product of complete symmetric functions and Hall-Littlewood functions
		\begin{align}
		\label{eq:general-q-supernomial}
			S_{\xi,\mu}(q) = \sum_{\eta} K_{\eta,\xi} \cdot  K_{\eta,\mu}(q) .
		\end{align}
		An explicit form of $S_{\xi,\mu}(q)$ is proven in \cite[Proposition 5.1]{MR1663325}, where $\mu = (\mu_1 , \ldots , \mu_m)$ is a partition and $\xi \in \mathbf{Z}_+^n$ a composition such that $|\mu| = |\xi| = M$, as
		\begin{align}
		\label{eq:general-q-supernomial-explicit}
			S_{\xi,\mu}(q) = \sum_{\{\nu\}} q^{\phi(\{\nu\})} \prod_{\substack{ 1 \leq a \leq n-1 \\ 1 \leq i \leq \mu_1 \ }} \qbinom{\nu_i^{(a+1)} - \nu_{i+1}^{(a)}}{\nu_i^{(a)} - \nu_{i+1}^{(a)}}_q
			,
		\end{align}
		with
		\[
			\phi(\{\nu\}) = \sum_{a=0}^{n-1} \sum_{i=1}^{\mu_1} \binom{\nu_i^{(a+1)} - \nu_i^{(a)}}{2}
			,
		\]
		and where the sum $\sum_{\{\nu\}}$ is indexed over the sequences of Young diagrams $\nu^{(1)},\ldots,\nu^{(n-1)}$ such that
		\begin{align*}
			& \emptyset \subset \nu^{(0)} \subset \nu^{(1)} \subset \cdots \subset \nu^{(n-1)} \subset \nu^{(n)} = \mu^t , \\
			& |\nu^{(a)}| = \xi_1 + \cdots + \xi_a \mbox{ for } 1 \leq a \leq n-1 .
		\end{align*}
		
		\begin{rem}
		\label{rem:keinelust}
			For $n=2$ and arbitrary $\mu$, \eqref{eq:general-q-supernomial} agrees with the definition of
$q$-supernomials as given by Schilling and Warnaar \cite{MR1665322}. See \cite[\S 3.1]{MR1768934} for a detailed
discussion.
		\end{rem}

		We define a slight variant of the above $q$-supernomials, which describes the string functions in the fusion product of $\mathfrak{sl}_{r+1}$ symmetric power representations $\mathcal{F}_\mu = V_{\mu_1 \omega_1} \ast V_{\mu_2 \omega_1} \ast \ldots \ast V_{\mu_m \omega_1}$. That is,
		\begin{align}
		\label{eq:general-q-supernomial-variant}
			S^\ast_{\xi,\mu}(q) = q^{n(\mu)} S_{\xi,\mu}(q^{-1}) = q^{n(\mu)} \sum_{\eta} K_{\eta,\xi} \cdot K_{\eta,\mu}(q^{-1}) ,
		\end{align}
		where for the partition $\mu = (\mu_1 , \ldots , \mu_m)$ we set $n(\mu)$ (Cf. \cite[(3.10)]{MR2047177}, \cite[\S 2.1]{0982.05105}, \cite[\S 2.1]{MR1665322}) to be the normalization constant
		\begin{align}
		\label{eq:the-mysterious-n-mu}
			n(\mu) = \sum_{i=1}^m (i-1) \mu_i = \sum_{1\leq i<j\leq m} \min (\mu_i,\mu_j) .
		\end{align}
		Note that this normalization ensures that $q^{n(\mu)}K_{\eta,\mu}(q^{-1})$ is a polynomial in $q$.

		Then, we have a fermionic formula (a positive sum of products of $q$-binomial coefficients) for the graded character of the above fusion product $\mathcal{F}_\mu$. That is,
		\begin{prp}
		\label{prp:string-functions-general}
			Let $\mu = (\mu_1 , \ldots , \mu_m)$ be a partition of $M$. Then,
			\[
			\chi (\mathcal{F}_\mu) = \sum_{\xi \text{ weight}} S^\ast_{\xi,\mu}(q) \cdot m_\xi .
			\]
		\end{prp}
		\begin{proof}
			Let $m_\xi$ denote the monomial symmetric functions. Then, with $\widetilde{K}_{\eta,\mu}(q) =  q^{n(\mu)} K_{\eta,\mu}(q^{-1})$ where $n(\mu) = \sum_i (i-1) \mu_i$ as in \cite[(3.10)]{MR2047177} one has \cite[Corollary 7.6]{MR2047177}:
			\begin{align*}
				\chi (\mathcal{F}_\mu) & = \sum_{\eta~\vdash M} \chi (\pi_\eta) \cdot \widetilde{K}_{\eta,\mu}(q) \\
					& = \sum_{\eta~\vdash M} s_\eta \cdot \widetilde{K}_{\eta,\mu}(q) \\
					& = \sum_{\eta~\vdash M} \bigg( \sum_{\xi \text{ weight}} K_{\eta,\xi} \cdot m_\xi \bigg) \cdot \widetilde{K}_{\eta,\mu}(q) \\
					& = \sum_{\xi \text{ weight}} \bigg( \sum_{\eta~\vdash M} K_{\eta,\xi} \cdot \widetilde{K}_{\eta,\mu}(q) \bigg) \cdot m_{\xi} \\
					& = \sum_{\xi \text{ weight}} S^\ast_{\xi,\mu}(q) \cdot m_{\xi} 
			\end{align*}
			Note that all partitions except $\mu$ have at most $r$ entries, corresponding to the rank of the Lie algebra.
		\end{proof}

		\begin{rem}
			For the graded character of fusion of fundamental representations $\ast_j V(\omega_{i_j})$, Chari and Loktev prove an equivalent fermionic formula \cite[Proposition 2.1.4]{MR2271991}.
		\end{rem}

		\begin{rem}[Cf.~\cite{bk11flags,janson,k11galoiskacmoody,v10}]
		\label{rem:flags}
			Kirillov \cite{MR1768934,0982.05105} is a great source of various combinatorial, geometric and statistical interpretations of $q$-supernomials $S_{\xi,\mu}(q)$. Let us shortly remark on the geometric one. As pointed out by Kirillov \cite[\S 1.4]{MR1768934} it has been proven by Shimomura \cite{0413.20037} that the $q$-supernomials count the number of rational points $\flag_\xi^\mu (\mathbf{F}_q)$ over the finite field $\mathbf{F}_q$ of the unipotent partial flag variety $\flag_\xi^\mu$.
			To be precise, for a composition $\xi \in \mathbf{Z}_+^r$ of $n$, a $\xi$-flag in a $n$-dimensional vector space $V$ is a sequence $V_1 \subset \cdots \subset V_r$ such that $\dim V_i = \xi_1 + \cdots + \xi_i$. The set of all such flags is the partial flag variety $\flag_\xi$. We let $\flag_\xi^\mu \subset \flag_\xi$ be the subset of the partial flag variety $\flag_\xi$ consisting of the set of all $\xi$-flags $F \in \flag_\xi$ fixed by a unipotent endomorphism $u \in \mathrm{Gl}(V)$ of type $\mu$ (a partition of $n$ that describes the Jordan canonical form of $u$). Then, $\flag_\xi^\mu$ is a closed subvariety of $\flag_\xi$, the so-called unipotent partial flag variety. Now, Shimomura \cite{0413.20037} proves that the $q$-supernomials count the number of $\mathbf{F}_q$-rational points in $\flag_\xi^\mu$. That is, with $n(\mu)$ as in \eqref{eq:the-mysterious-n-mu} one has
			\begin{align}
				\# \flag_\xi^\mu (\mathbf{F}_q) = q^{n(\mu)} S_{\xi,\mu}(q^{-1}) = S^\ast_{\xi,\mu}(q).
			\end{align}
			In particular, the basic specialization of the fusion module $\mathcal{F}_\mu$ gives the number of $\mathbf{F}_q$-rational points in $\coprod_\xi \flag_\xi^\mu$, 
			\begin{align}
				\chi (\mathcal{F}_\mu)(q) = \sum_\xi S^\ast_{\xi,\mu}(q) = \sum_\xi \# \flag_\xi^\mu (\mathbf{F}_q) .
			\end{align}
		\end{rem}
		
		Our \autoref{prp:string-functions-general} exhibits the objects that have to be analyzed in order to
establish a general central limit theorem along the same lines as \autoref{thm:central-limit-basic-spec-fusion}. The
explicit expression \eqref{eq:general-q-supernomial-variant} shows that one can interpret the $q$-supernomials again as
mixtures of probability distributions. For an Ansatz let
		\begin{align*}
		  f_{\mu,\eta}(q) & = \frac{q^{n(\mu)} K_{\eta,\mu}(q^{-1})}{K_{\eta,\mu}} , \\
		  \Prob (X_{\mu,\xi} = \eta) & = \frac{K_{\eta,\xi}K_{\eta,\mu}}{\sum_{\eta} K_{\eta,\xi}K_{\eta,\mu}} .
		\end{align*}
		Here, $f_{\mu,\eta}(q)$ would immitate the inversion statistic, and $X_{\mu,\xi}$ the mixture
distribution. It should be straightforward to check the reductions to the distributions investigated in
\autoref{sec:distributions} in the case of $q$-supernomials as defined by Schilling and Warnaar (see
\autoref{rem:keinelust}). We pose a conjecture for further research.
	\begin{cnj}
	\label{cnj:fusion}
	  Consider the sequence of fusion modules of symmetric power representations for the current algebra $\mathfrak{sl}_r \otimes \mathbf{C}[t]$
	  \[\mathcal{F}_{\mu^N} = V_{\omega_1}^{\ast \Ln_1} \ast V_{2 \omega_1}^{\ast \Ln_2} \ast
\ldots \ast V_{m \omega_1}^{\ast \Ln_m}, \]
	  associated to the partition
	  $\mu^N = (1^{\Ln_1},2^{\Ln_2}, \ldots ,
m^{\Ln_m})$
	  with $\Ln_i$-many $i$'s. Assume that as $N \ra \infty$ we have
	  \[\frac 1N (\Ln_1,\Ln_2, \ldots , \Ln_m) \ra \va \neq 0. \]
	  Then, the central string functions and the basic specialization of the character $\chi(\mathcal{F}_{\mu^N})$ behaves
asymptotically normal as $N \ra \infty$.
	\end{cnj}	
	
	\section{Acknowledgements}
	
	\noindent The first author would like to thank Evgeny Feigin for helpful discussions.

	\bibliographystyle{amsplain}
	\bibliography{qsupernomials}
	
\end{document}